\documentclass[11pt]{amsart}
\usepackage{amsmath,amsfonts,amssymb,amsthm}
\usepackage{mathtools}
\usepackage[a4paper, total={8.5in, 8.5in}]{geometry}
\usepackage[table]{xcolor}
\usepackage{enumitem}

\usepackage[utf8]{inputenc}
\usepackage[bottom]{footmisc}

\usepackage{here}

\newtheorem{theorem}{Theorem}[section]
\newtheorem{lemma}[theorem]{Lemma}
\newtheorem{prop}[theorem]{Proposition}

\newtheorem{conjecture}[theorem]{Conjecture}

\theoremstyle{definition}

\newtheorem{remark}[theorem]{Remark}

\makeatletter
  
  \@addtoreset{equation}{section}
   \makeatother
   
\begin{document}

\title[Multiple zeta-functions]{On the behavior of multiple zeta-functions
with identical arguments on the real line I}

\author{Kohji Matsumoto}
\address{K. Matsumoto: Graduate School of Mathematics, Nagoya University, Chikusa-ku, Nagoya 464-8602, Japan}
\email{kohjimat@math.nagoya-u.ac.jp}

\author{Ilija Tanackov}
\address{I. Tanackov: Faculty of Technical Sciences, University of Novi Sad,
Trg Dositeja Obradovi{\'c}a 6, 21000 Novi Sad, Serbia}
\email{ilijat@uns.ac.rs}

\thanks{
Research of the first author is
supported by Grants-in-Aid for Science Research no. 18H01111, JSPS, and that of the
second author is by Ministry of Science and Technological Development of Serbia 
no. TR 36012.}

\begin{abstract}
In the present series of papers,
we study the behavior of $r$-fold zeta-functions of Euler-Zagier type with identical
arguments $\zeta_r(s,s,\ldots,s)$ on the real line.    
In this first part, we consider the behavior on the interval $[0,1]$.
Our basic tool is an
``infinite'' version of Newton's classical identities.    We carry out numerical
computations, and draw graphs of $\zeta_r(s,s,\ldots,s)$ for real $s\in [0,1]$, for several
small values of $r$.    Those graphs suggest various properties of 
$\zeta_r(s,s,\ldots,s)$, some of which we prove rigorously.    For example, we show that
$\zeta_r(s,s,\ldots,s)$ has $r$
asymptotes at $\Re s=1/k$ ($1\leq k\leq r$), and
determine the asymptotic behavior of $\zeta_r(s,s,\ldots,s)$ close to those asymptotes.   
Until now, the existence of one real zero of $\zeta_2(s,s)$ has been known.
Our present computations establish several new real zeros between asymptotes
in the cases $3\leq r\leq 10$.
Moreover, on the number of real zeros of $\zeta_r(s,s,\ldots,s)$,
we raise a conjecture, and a formula for calculating the number of zeros on the interval
$[0,1]$ is derived.
\end{abstract}

\maketitle

\section{Introduction}

The Euler-Zagier multiple zeta-function
\begin{align}\label{EZ_def}
\zeta_r (s_1,s_2,...,s_r)= \sum_{1 \leq m_1 <m_2<...<m_r } \frac{1}{m_1^{s_1}} \cdot\frac{1}{m_2^{s_2}} \cdots\frac{1}{m_r^{s_r}}, 
\end{align}
where $s_1,\ldots,s_r$ are complex variables, has been studied extensively in
recent decades.   The series \eqref{EZ_def} is convergent absolutely when $\Re s_j$
($1\leq j\leq r$) are sufficiently large, but it can be continued meromorphically to
the whole space $\mathbb{C}^r$ (see, for example, \cite{AET01}).
Analytic properties of $\zeta (s_1,s_2,...,s_r)$ have been studied in a lot of papers,
among which we mention here a numerical study on $\zeta_2(s_1,s_2)$ by the
first-named author and M. Sh{\=o}ji \cite{MatSho14} \cite{MatSho20}.    In \cite{MatSho14}
the case $s_1=s_2=s$ was treated, and the general two-variable case was discussed in
\cite{MatSho20}.
In particular, in \cite{MatSho14} it has been shown that the distribution of the zeros of
$\zeta_2(s,s)$ is not similar to that of the Riemann zeta-function $\zeta(s)=\zeta_1(s)$
(especially the Riemann hypothesis does not hold), but rather, has a resemblance to
the distribution of the zeros of Hurwitz zeta-functions
$\zeta(s,\alpha)=\sum_{m=0}^{\infty}(m+\alpha)^{-s}$ ($0<\alpha\leq 1$).   

It is desirable to generalize the study \cite{MatSho14} \cite{MatSho20} to the
general $r$-fold case.     It is natural to consider first the case when all variables
are identical: $s_1=\cdots=s_r=s$.    We write
$$
\zeta_r(s)=\zeta_r(s,s,\ldots,s).
$$
In \cite{MatSho14}, two topics were considered; the distribution of zeros of
$\zeta_2(s)$ in the complex plane $\mathbb{C}$, and the behavior of $\zeta_2(s)$ on the
real axis $\mathbb{R}$.    The aim of the present series of papers is to study the 
behavior of $\zeta_r(s)$ ($r\geq 2$) on $\mathbb{R}$.    
In this first part, we discuss the behavior of $\zeta_r(s)$ on the interval $[0,1]$.
The behavior outside this interval will be considered in the second part \cite{MMT}.
Moreover we have a plan of developing a study on $\zeta_r(s)$ for 
$s\in\mathbb{C}\setminus\mathbb{R}$ in the near future.

Our investigation is based on numerical computations on the behavior of
$\zeta_r(s)$ for $s\in\mathbb{R}$.   
We draw the graphs of $\zeta_r(s)$ for $2\leq r\leq 10$, which
suggest various properties of 
$\zeta_r(s,s,\ldots,s)$, some of which we prove rigorously.    For example, we show that
$\zeta_r(s,s,\ldots,s)$ has $r$
asymptotes at $\Re s=1/k$ ($1\leq k\leq r$), and
determine the asymptotic behavior of $\zeta_r(s,s,\ldots,s)$ close to those asymptotes.   
We also prove asymptotic formulas for $\zeta_r(-k,-k,\ldots,-k)$, where $k$ takes odd
positive integer values and tends to $+\infty$.

Moreover, numerical computations give some insight on the number of real zeros of 
$\zeta_r(s,s,\ldots,s)$.
The existence of one real zero of $\zeta_2(s,s)$ has already been reported in
\cite{MatSho14}.
Our present computations establish several new real zeros between asymptotes
in the cases $3\leq r\leq 10$.
Moreover, on the number of real zeros of $\zeta_r(s,s,\ldots,s)$,
we raise a conjecture, and a formula for calculating the number of zeros on the interval
$[0,1]$ is derived.

As we mentioned above, there is some similarity between the behavior of $\zeta_2(s)$ 
and that of Hurwitz zeta-functions.   Such similarity can also be expected for
$\zeta_r(s)$, $r\geq 3$.     Therefore the study on the real zeros of Hurwitz
zeta-functions can be suggestive in our research.

Recently there has been big progress on the study of real zeros of Hurwitz
zeta-functions (see Schipani \cite{Sch11}, Nakamura \cite{Nak16a} \cite{Nak16b},
Matsusaka \cite{Mat18}, and Endo and Suzuki \cite{EndSuz19}). 
It is an interesting problem to search for analogues of those studies in the case of
$\zeta_r(s)$.     It is to be mentioned that the idea included in those articles
was already applied by Nakamura himself \cite{Nak16b} to a variant of $\zeta_2(s)$ of
Hurwitz-Lerch type, and by Sakurai \cite{SakPre} to Barnes double zeta-functions.

\section{Newton's identities}\label{sec-2}

In this section we prepare the basic identities among multiple zeta-functions, based 
on the classical identities of
Newton, which we will use in our computations.    

Let $\mathbb{N}$ be the set of positive integers.
The polynomial of degree $n\in\mathbb{N}$, with roots $x_1,\ldots,x_n$ may be written as
\begin{align}
\prod_{m=1}^n(x-x_m)=\sum_{r=0}^n (-1)^r e_r x^{n-r},
\end{align}
where $e_r$ are symmetric polynomilas given by
\begin{align}
e_r=e_r(x_1,x_2,\ldots,x_n)=\sum_{1\leq m_1<m_2<\cdots<m_r\leq n}x_{m_1}x_{m_2}
\cdots x_{m_r}\qquad (1\leq r\leq n).
\end{align}
For example $e_1=x_1+x_2+\cdots+x_n$, $e_2=\sum_{1\leq i<j\leq n}x_i x_j$, etc., 
and we interpret that $e_0=1$.

Define the $r$-th power sum
$$
p_r=p_r(x_1,x_2,\ldots,x_n)=x_1^r+x_2^r+\cdots+x_n^r.
$$
Newton's identities are given by the following statement:
\begin{align}\label{Newton_id}
re_r(x_1,x_2,\ldots,x_n)=\sum_{j=1}^r (-1)^{j-1}e_{r-j}(x_1,x_2,\ldots,x_n)
p_j(x_1,x_2,\ldots,x_n),
\end{align}
where $r,n\in\mathbb{N}$ with $r\leq n$.
This is due to Sir Issac Newton.   Various proofs can be found in, for example,
Zolberger \cite{Zol84}, Kalman \cite{Kal00}, and Mukherjee and Bera \cite{MukBer19}.

Now we let $x_m=m^{-s}$, where $s$ is a complex variable, and put
\begin{align}\label{N_r_def}
N_r(s)=e_r(1^{-s},2^{-s},\ldots,n^{-s})=\sum_{1\leq m_1<m_2<\cdots<m_r\leq n}
(m_1 m_2\cdots m_r)^{-s} \qquad (1\leq r\leq n).
\end{align}
We take the limit $n\to\infty$.    When $\Re s>1$, the limit of the right-hand side
converges, and is equal to $\zeta_r(s)$.    Therefore,
\begin{align}
\lim_{n\to\infty}N_r(s)=\zeta_r(s) \qquad (\Re s>1).
\end{align}
Also we see that
\begin{align}
\lim_{n\to\infty}p_j(1^{-s},2^{-s},\ldots,n^{-s})=\sum_{m=1}^{\infty}m^{-js}=\zeta(js),
\end{align}
where $\zeta(s)=\zeta_1(s)$ is the Riemann zeta-function.
Therefore, taking the limit $n\to\infty$ of Newton's identity \eqref{Newton_id}, 
we obtain
\begin{align}\label{zeta_id}
r\zeta_r(s)=\sum_{j=1}^r (-1)^{j-1}\zeta_{r-j}(s)\zeta(js) \qquad (r\in\mathbb{N}),
\end{align}
where we understand that $\zeta_0=1$.
This identity is first valid for $\Re s>1$, but then, by the meromorphic 
continuation, it can be extended to any $s\in\mathbb{C}$.

This is not a new identity.
In fact, this is essentially the well-known harmonic product formula, and Kamano \cite{Kam06}
deduced \eqref{zeta_id} (in a little more generalized form) from the harmonic product
formula.    However our argument is different.

For several small values of $r$, this formula implies:
\begin{align}\label{2_id}
\zeta_2(s)=\frac{1}{2}\left\{\zeta(s)^2-\zeta(2s)\right\},
\end{align}
\begin{align}\label{3_id}
\zeta_3(s)=\frac{1}{3}\left\{\zeta_2(s)\zeta(s)-\zeta(s)\zeta(2s)+\zeta(3s)\right\},
\end{align}
\begin{align}\label{4_id}
\zeta_4(s)=\frac{1}{4}\left\{\zeta_3(s)\zeta(s)-\zeta_2(s)\zeta(2s)+\zeta(s)\zeta(3s)
-\zeta(4s)\right\}.
\end{align}
From these formulas, it is also possible to get expressions of $\zeta_r(s)$ only in
terms of the Riemann zeta-function.   Substituting \eqref{2_id} into \eqref{3_id}, we
obtain
\begin{align}\label{3_riemann}
\zeta_3(s)=\frac{1}{6}\left\{\zeta(s)^3-3\zeta(s)\zeta(2s)+2\zeta(3s)\right\}.
\end{align}
Similarly,
\begin{align}\label{4_riemann}
\zeta_4(s)=\frac{1}{24}\left\{\zeta(s)^4-6\zeta(s)^2\zeta(2s)+3\zeta(2s)^2
+8\zeta(s)\zeta(3s)-6\zeta(4s)\right\}.
\end{align}

Consider a formal infinite polynomial 
\begin{align}\label{inf_poly}
 \prod_{m=1}^{ \infty } (x- x_m)= \lim_{n\to \infty}  \left\{( x-x_1)(x- x_2)\cdots
 (x- x_n )\right\}.
 \end{align}
When $x_m=m^{-s}$, the right-hand side is equal to
\begin{align}
\lim_{n\to \infty}\left\{x^n-x^{n-1} \sum_{m_1=1}^{n} \frac{{1} }{m_1^s}+x^{n-2}  \sum_{1 \leq m_1<m_2\leq n} \frac{{1} }{(m_1m_2)^s} -x^{n-3}\sum_{1 \leq m_1<m_2<m_3\leq n} \frac{{1} }{(m_1m_2m_3)^s}+...+\frac{{(-1)^n} }{{(n!)^s}} \right\},
\end{align}
whose each coefficient tends to $\zeta_r(s)$ as $n\to\infty$.
It should be noted that the above expansion is based on the rules of Vieta (F. Vi{\`e}te)
for the infinite polynomial \eqref{inf_poly}.
This observation shows that our fundamental formula \eqref{zeta_id} may be
formally regarded as Newton's identities for the infinite polynomial
\eqref{inf_poly}.

The above argument was inspired by the second-named author's
"new-nacci" method for solving polynomial roots.    This method is based on the convergence of successive
Fibonacci-type sequences, 
which are Newton's identities (see Tanackov et al.~\cite{Tan20}).


\section{The double and the triple zeta-functions}\label{sec-3}

Before going into the discussion of general $r$-fold situation, in this section
we observe the behavior of $\zeta_r(s)$ on the real line for $r=2$ and 3.
Here we give the graphs not only on the interval $[0,1]$, but also outside this
interval.   

\subsection{The double case}

The analytic properties of the double zeta-function are well studied (see Matsumoto \cite{Mat03} \cite{Mat04}, Kiuchi et al. \cite{KiuTan06} \cite{KTZ11}, Matsumoto and
Tsumura \cite{MatTsu15}, etc.).
The double zeta-function with identical arguments $\zeta_2(s)$, $s\in\mathbb{R}$,
can be computed by \eqref{2_id} (see Fig \ref{Fig1}).

\begin{figure}[h]
\centering
\includegraphics{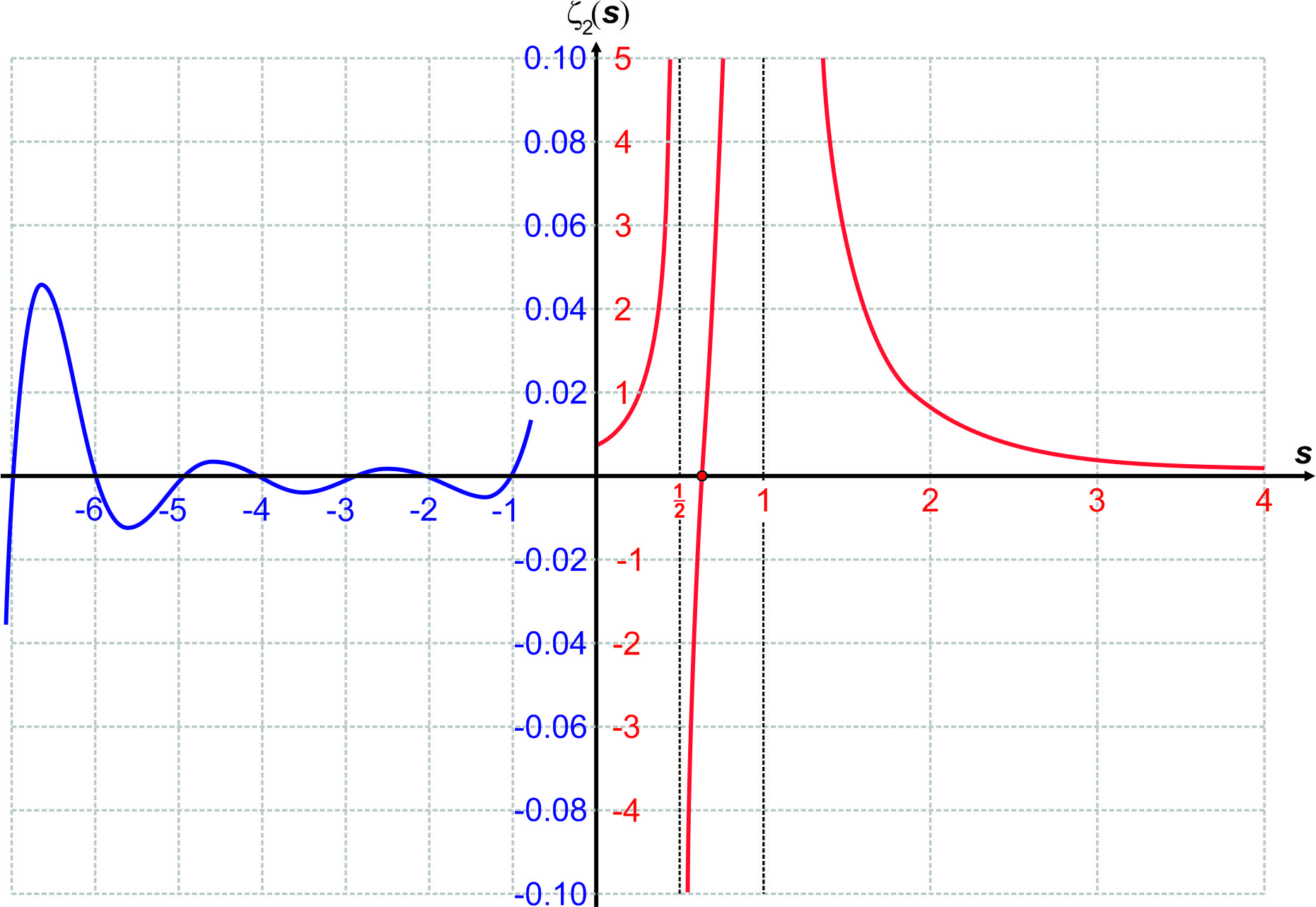}
\caption{The double zeta-function (Note that the vertical scale in the negative half-plane
is different from that in the positive half-plane)}
\label{Fig1}
\end{figure}

The graph of the double zeta-function has two vertical asymptotes $\Re s=1$ and
$\Re s=1/2$ in the positive half-plane.    The first one is caused by the factor
$\zeta(s)$, while the second one by $\zeta(2s)$, respectively.
One zero \(\zeta_2(0.6268175...)=0\) appears between these two vertical asymptotes.   We call this type of zero an ``inter-asymptotic'' zero, and write IAZ for brevity.

All the 
stated values, the graph, the location of zeros for $\zeta_2(s)$ mentioned above were analyzed in detail by the first-named author and Sh{\=o}ji \cite{MatSho14}.

\subsection{The triple case}

Unlike the double zeta, the analytical properties of the triple zeta-function have been more modestly
investigated (see Kiuchi and Tanigawa \cite{KiuTan08}).
The values of the triple zeta-function have been discussed for
different arguments of positive integer values (Markett \cite{Mar94}, 
Hoffman and Moen \cite{HofMoe96}, and Machide \cite{Mac13}).
%


\begin{figure}[h]
\centering
\includegraphics{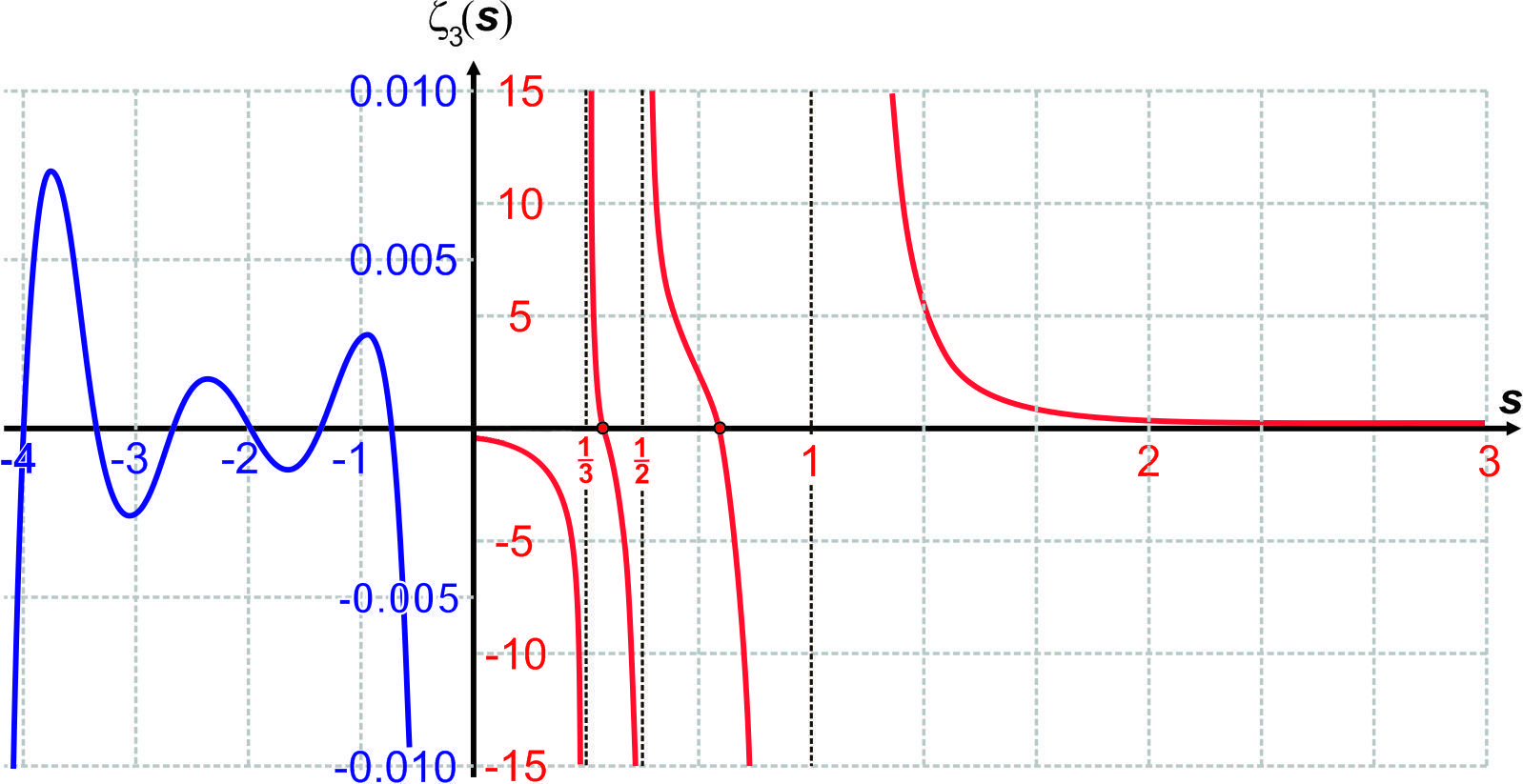}
\caption{The triple zeta-function (Note that the vertical scale in the negative half-plane
is different from that in the positive half-plane)}
\label{Fig3}
\end{figure}

We compute $\zeta_3(s)$ for $s\in\mathbb{R}$ by using \eqref{3_id}, \eqref{3_riemann}
(see Fig \ref{Fig3}).
In the positive half-plane, the graph of the triple zeta-function has three vertical asymptotes: $\Re s=1, 1/2$ and $1/3$.
Vertical asymptotes $\Re s=1$ and $1/2$ are inherited from the double zeta-function,
while the new vertical asymptote $\Re s=1/3$ is coming from the factor $\zeta(3s)$ 
in \eqref{3_id}, \eqref{3_riemann}.
There are two IAZs between the vertical asymptotes \(\zeta_3(+0.385782)=0\) and
\(\zeta_3(+0.724902)=0\).

%


 
\section{The cases $r=4$, 5 and 6}

Now we concentrate upon the behavior of multiple zeta-functions on the interval $[0,1]$.
First, in this section, we present the graphs of $\zeta_r(s)$, $s\in [0,1]$, for
$r=4,5$ and $6$.
 
The quadruple zeta-function was studied, for example, by Machide \cite{Mac19}.
The theory of the fourth power mean of the Riemann zeta-function
(Motohashi \cite{Mot93}, Ivi{\'c} and Motohashi \cite{IviMot95}) is somewhat relevant.

The quadruple zeta-function $\zeta_4(s)$ has four 
asymptotes: $\Re s=1, 1/2, 1/3$ and $1/4$, and
has four IAZs in the interval \( s\in [0,1]\) 
(see Fig \ref{Fig6}):

\begin{itemize}[noitemsep]
    \item One IAZ $\in (1/4,1/3)$:\;\(\zeta_4(0,27886...)\approx 0.\)
    
    \item One IAZ $\in (1/3,1/2)$:\;\(\zeta_4(0,387072...)\approx 0 .\)
    \item Two IAZs $\in (1/2,1)$: \(\zeta_4(0,571348...) \approx 0 \) and \(\zeta_4(0,783444...)\approx 0. \)

\end{itemize}

The quadruple zeta function \(\zeta_4(s)\) has one minimum \(\zeta_4(0,693658...) \approx -4,0699572... \) between vertical
asymptotes $1/2$ and 1.

There are two new features here:

(i) When $s\to 1/2$, the value $\zeta_4(s)$ tends to $+\infty$ for the both of the limits
$s\to 1/2+0$ and $s\to 1/2-0$.    The reason is that the pole of $\zeta_4(s)$ at $s=1/2$
is of order 2, because of the term $3\zeta(2s)^2$ in \eqref{4_riemann}.

(ii) In the interval $(1/2,1)$ there are two zeros.


\begin{figure}[h]
\centering
\includegraphics{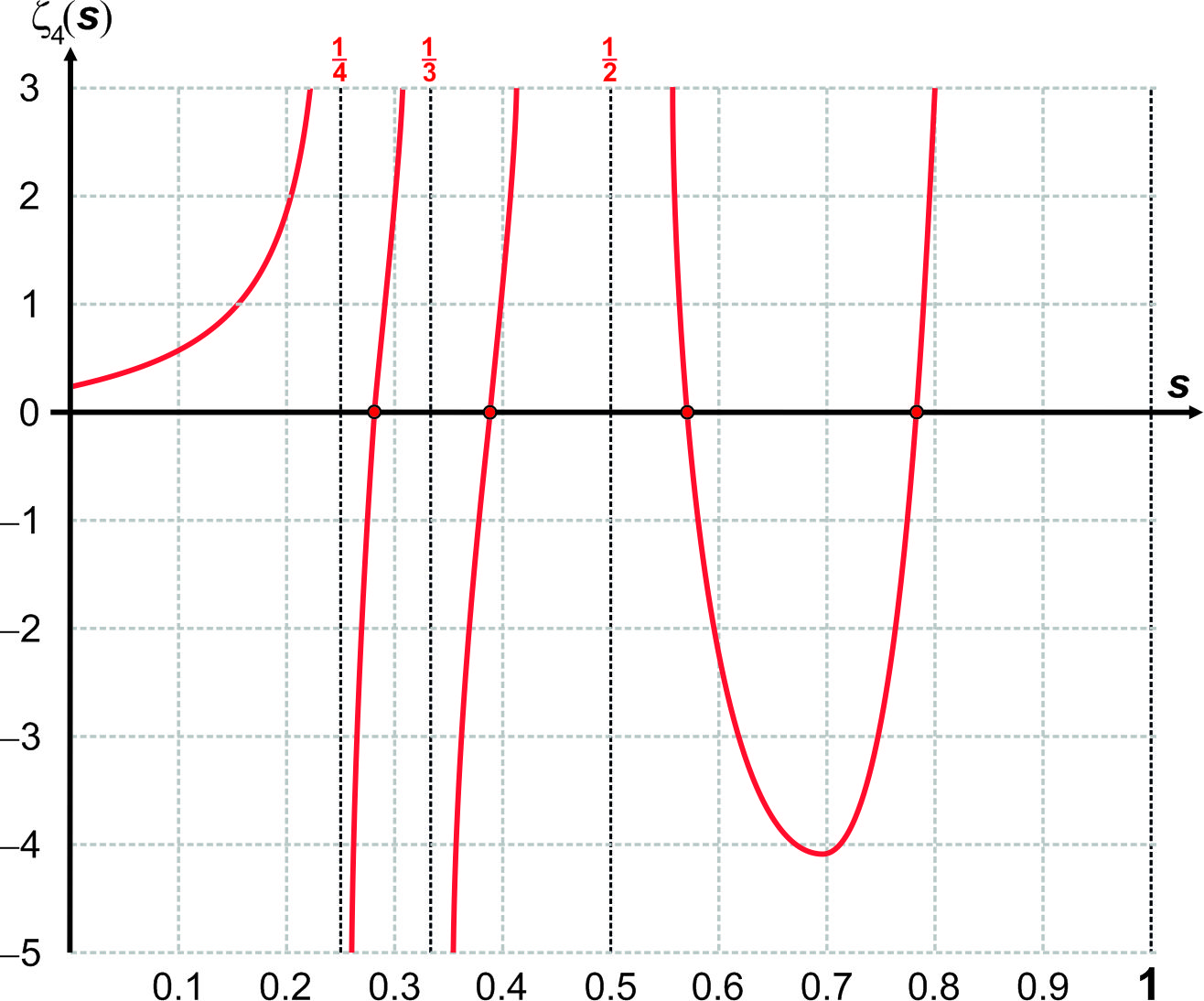}
\caption{The quadruple zeta-function for $s\in [0,1]$}
\label{Fig6}
\end{figure}

Next,
the five-fold zeta-function $\zeta_5(s)$ has five asymptotes: $\Re s=1, 1/2, 1/3, 1/4$ 
and $1/5$, and has five IAZs in \( s\in [0,1]\) (Fig \ref{Fig7}):

\begin{itemize}[noitemsep]
    \item One IAZ $\in (1/5,1/4)$: \(\zeta_5(0,218315...)\approx 0.\)
    
    \item One IAZ $\in (1/4,1/3)$: \(\zeta_5(0,278346...)\approx 0.\)
    \item One IAZ $\in (1/3,1/2)$: \(\zeta_5(0,423505...) \approx 0. \) 
    \item Two IAZs $\in (1/2,1)$: \(\zeta_5(0,643861...)\approx 0 \) and \(\zeta_5(0,881698...)\approx 0.\)
 \end{itemize}
Also \(\zeta_5(s)\) has one maximum \(\zeta_5(0,776027...)\approx  +6,003808... \) between vertical asymptotes $1/2$ and 1.
We see that $\zeta_5(s)\to -\infty$ as $s\to 1/2\pm0$.    

\begin{figure}[h]
\centering
\includegraphics{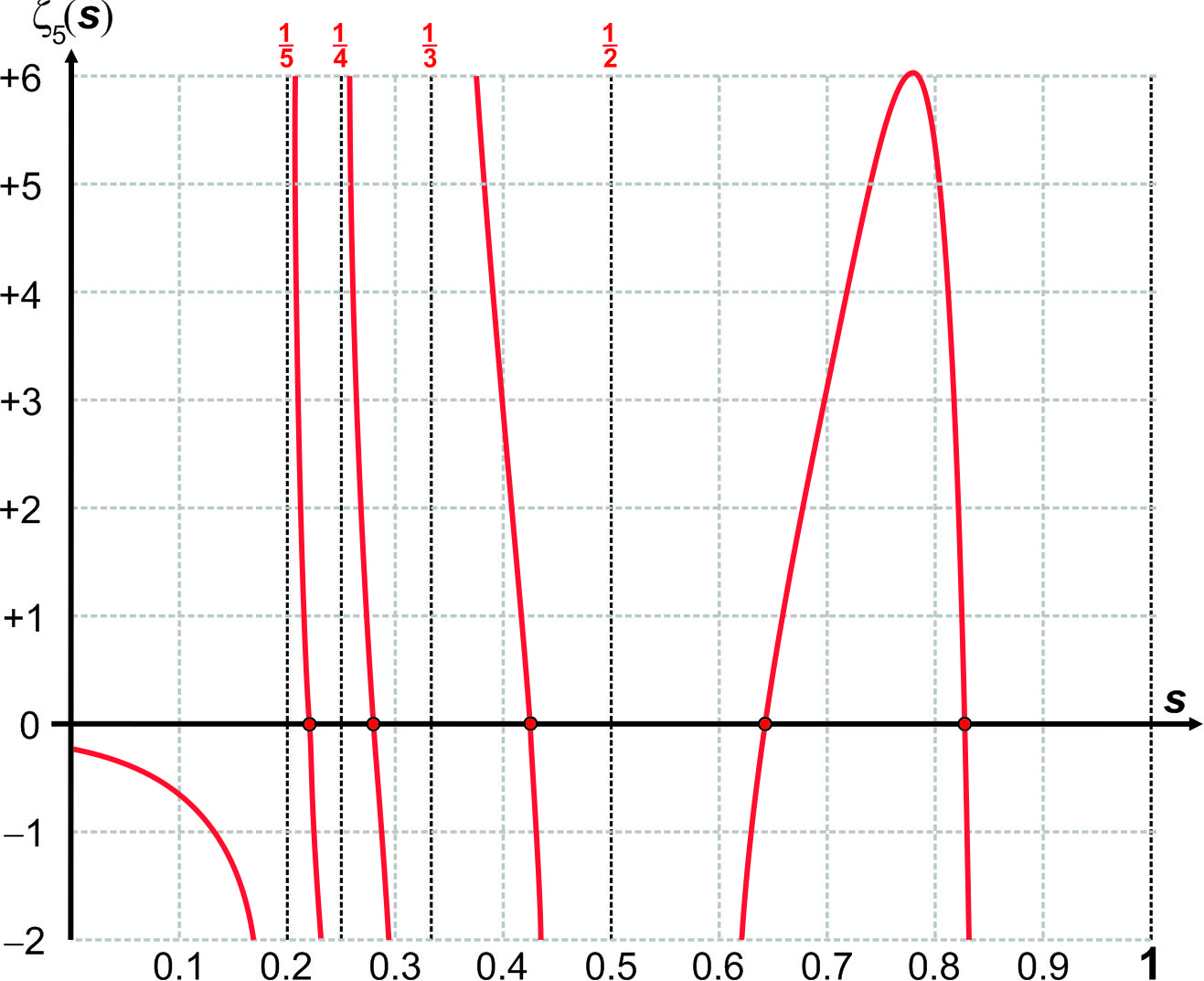}
\caption{The five-fold zeta-function for $s\in [0,1]$}
\label{Fig7}
\end{figure}

The six-fold zeta-function has six asymptotes: $\Re s=1, 1/2, 1/3, 1/4, 1/5$ 
and  $1/6$, and has eight IAZs in \( s\in [0,1]\) (Fig \ref{Fig8}): 

\begin{itemize}[noitemsep]
    \item One IAZ $\in (1/6,1/5)$: \(\zeta_6(0,179347...)\approx 0.\)
    
    \item One IAZ $\in (1/5,1/4)$: \(\zeta_6(0,217682...)\approx 0.\)
    \item  One IAZ $\in (1/4,1/3)$: \(\zeta_6(0,279817...) \approx0 \).
    \item Two IAZs $\in (1/3,1/2)$: \(\zeta_6(0,362716...)\approx 0 \) and \(\zeta_6(0,419205...)\approx 0. \)
    
    \item Three IAZs $\in (1/2,1)$: \(\zeta_6(0,549629...)\approx 0 \), \(\zeta_6(0,696745...)\approx 0 \) and
\(\zeta_6(0,848546...)\approx 0.\) 
\end{itemize}

The six-fold zeta-function \(\zeta_6(s)\) has two minimums and one maximum in \( s\in [0,1]\) :

\begin{itemize}[noitemsep]
    \item One minimum between vertical asymptotes $1/3$ and $1/2$: \(\zeta_6(0,386562...)\approx  -2,462682...\)
    \item One maximum \(\zeta_6(0,578067...)\approx  +5,283455...\)and one minimum \(\zeta_6(0,818945...)\approx  -10,900018...\)
between vertical asymptotes $1/2$ and 1.
\end{itemize}

There again appear two new features.    The case $r=6$ of \eqref{zeta_id} gives
\begin{align}
\zeta_6(s)=\frac{1}{6}\left\{\zeta_5(s)\zeta(s)-\zeta_4(s)\zeta(2s)+\zeta_3(s)\zeta(3s)
-\zeta_2(s)\zeta(4s)+\zeta(s)\zeta(5s)-\zeta(6s)\right\},
\end{align}
from which we can see that the orders of the poles of $\zeta_6(s)$ at $s=1/3$ and
$s=1/2$ are 2 and 3, respectively.   This gives the behavior of $\zeta_6(s)$ around
the asymptotes $1/3$ and $1/2$, indicated in the figure.
Moreover, the interval $(1/2,1)$ now includes three IAZ, and the interval $(1/3,1/2)$
also includes more than one IAZ.

\begin{figure}[h]
\centering
\includegraphics{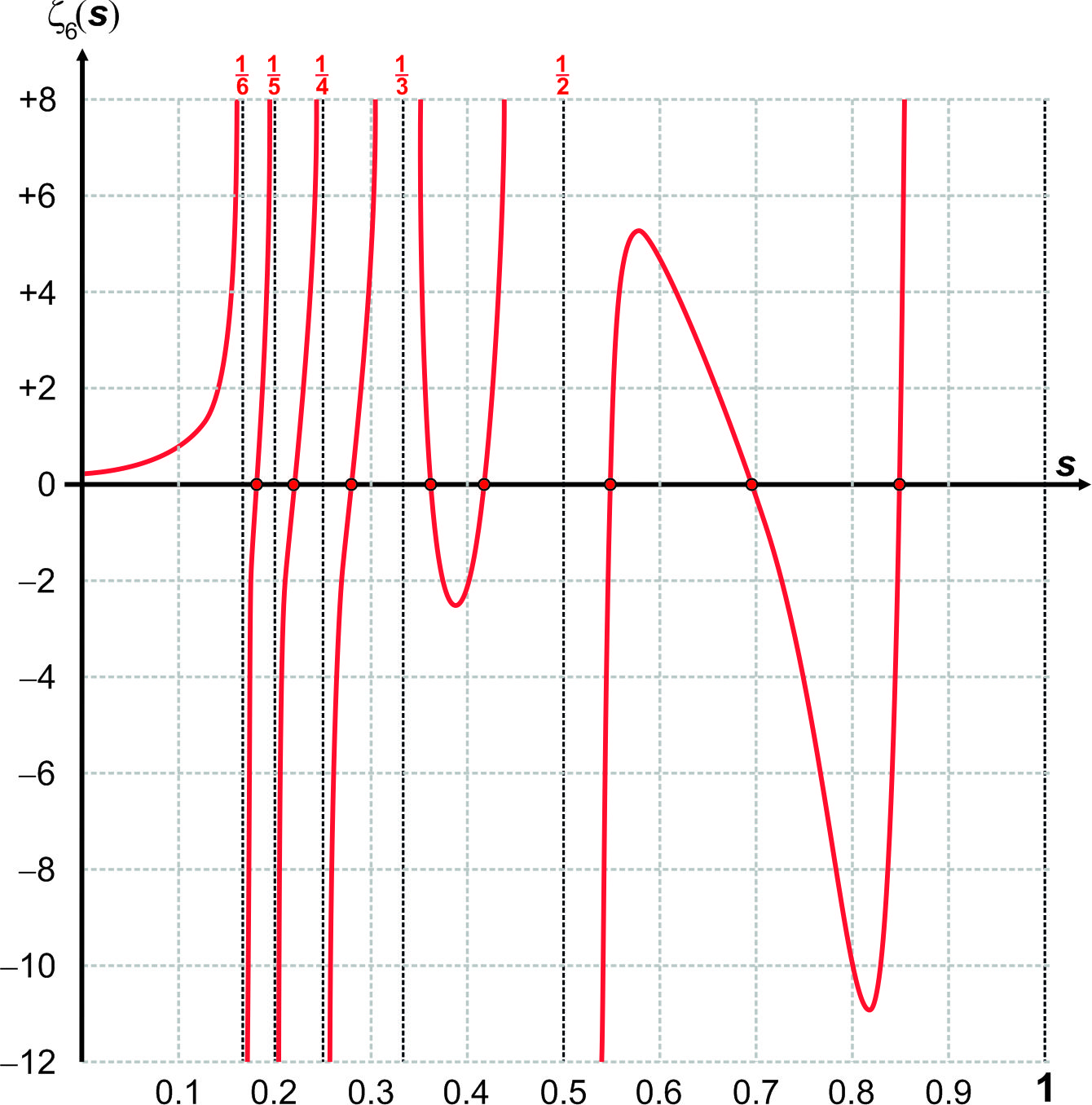}
\caption{The six-fold zeta-function for $s\in [0,1]$}
\label{Fig8}
\end{figure}
 
\section{Theorems and a conjecture}
 
The above numerical data suggests several properties of $\zeta_r(s)$, some of which we will
prove here.   The first result is:

\begin{theorem}\label{Th<1/r}
On the interval $0\leq s< 1/r$, the value of $\zeta_r(s)$ is positive for even $r$,
and negative for odd $r$.    In particular, there is no zero of $\zeta_r(s)$ in
this interval.
\end{theorem}    

\begin{proof}
When $r=1$, it is well-known that $\zeta(s)<0$ for $0\leq s<1$.   Then, from
\eqref{2_id} it is clear that $\zeta_2(s)>0$ for $0\leq s<1/2$.

In general, we prove the assertion by induction.
From \eqref{zeta_id} we have
\begin{align}\label{zeta_id2}
\zeta_r(s)=\frac{1}{r}\sum_{j=1}^{r-1}(-1)^{j-1}\zeta_{r-j}(s)\zeta(js)+
\frac{1}{r}(-1)^{r-1}\zeta(rs).
\end{align}
Let $0\leq s< 1/r$.    Then we have $\zeta(rs)<0$, hence the signature of the term 
$r^{-1}(-1)^{r-1}\zeta(rs)$ is given by $(-1)^r$.
Using the induction assumption, we have that $\zeta_{r-j}(s)$ is non-zero and its 
signature is
given by $(-1)^{r-j}$, so the signature of the term $(-1)^{j-1}\zeta_{r-j}(s)\zeta(js)$
is given by $(-1)^{(j-1)+(r-j)+1}=(-1)^r$.
Therefore all terms on the right-hand side of \eqref{zeta_id2} have the same signature
$(-1)^r$, and so the signature of $\zeta_r(s)$.
\end{proof}

Using Theorem \ref{Th<1/r}, we can prove the following theorem on
the asymptotic behavior of $\zeta_r(s)$ near the asymptotes.

\begin{theorem}\label{Th-asymp}
{\rm (i)} The function $\zeta_r(s)$ has poles only at $s=1/k$ $(1\leq k\leq r)$ of 
order $[r/k]$, where $[x]$ denotes the integer part of $x$.

{\rm (ii)} As $s\to 1/k$, the asymptotic behavior of $\zeta_r(s)$ is given by
\begin{align}\label{Th_formula}
\zeta_r(s)\sim C_r(k)(ks-1)^{-[r/k]},
\end{align}
where $C_r(k)$ is a non-zero real constant, whose singature coincides with the signature of
$(-1)^{r+[r/k]}$.    
\end{theorem}

\begin{proof}
When $r=1$, $\zeta_1(s)=\zeta(s)$ is the Riemann zeta-function, which has the only
pole at $s=1$ and this pole is simple with residue 1.    Therefore the theorem
for $r=1$ clearly holds. 

Now assume $r\geq 2$, and we prove the theorem by induction.
We use \eqref{zeta_id2} again.   On the right-hand side of \eqref{zeta_id2},
the factor $\zeta(js)$ has the only pole at $s=1/j$, while by induction assumption,
the poles of $\zeta_{r-j}(s)$ are at $s=1,1/2,\ldots,1/(r-j)$.
Therefore, for any fixed $k$ ($1\leq k\leq r$), the factors on the right-hand side of
\eqref{zeta_id2}, which are singular at $s=1/k$, are $\zeta(ks)$ and
$\zeta_{r-j}(s)$ with $j\leq r-k$.

Consider the case $k=r$.    The singularity at $s=1/r$ appears only in the last
term $r^{-1}(-1)^{r-1}\zeta(rs)$.    Therefore
\begin{align}
\zeta_r(s)\sim \frac{1}{r}(-1)^{r-1}\frac{1}{rs-1}
\end{align}
as $s \to 1/r$,
which implies \eqref{Th_formula} for $k=r$, with $C_r(r)=r^{-1}(-1)^{r-1}$.

Next, let $1\leq k\leq r-1$.
The asymptotic behavior of the singular factors at $s=1/k$ on the right-hand side of
\eqref{zeta_id2} can be written down by induction
assumption.    
For $j\leq r-k$ with $j\neq k$, we have
$$
\zeta_{r-j}(s)\zeta(js)\sim \zeta(j/k)C_{r-j}(k)(ks-1)^{-[(r-j)/k]}
$$ 
as $s\to 1/k$.
As for the term $\zeta_{r-k}(s)\zeta(ks)$, the factor $\zeta(ks)$ is always singular at 
$s=1/k$, while $\zeta_{r-k}(s)$ is singular only when $k\leq r-k$, that is $k\leq r/2$.
Therefore we may write
\begin{align}\label{zeta_asymp}
\zeta_r(s)\sim A(s)+B(s)
\end{align}
as $s\to 1/k$, where
\begin{align}\label{A_def}
A(s)=\frac{1}{r}\sum_{\stackrel{1\leq j\leq r-k}{j\neq k}}(-1)^{j-1}
\zeta\left(\frac{j}{k}\right)C_{r-j}(k)(ks-1)^{-[(r-j)/k]}
\end{align}
and
\begin{align}
B(s)=\left\{
\begin{array}{lll}
\displaystyle{\frac{1}{r}(-1)^{k-1}C_{r-k}(k)
(ks-1)^{-[(r-k)/k]}\frac{1}{ks-1}}  & {\rm if} & k\leq r/2,\\
\displaystyle{\frac{1}{r}(-1)^{k-1}\zeta_{r-k}\left(\frac{1}{k}\right)
\frac{1}{ks-1}}  & {\rm if}  &  k>r/2.
\end{array}
\right.
\end{align}
say.    Since
$(ks-1)^{-[(r-k)/k]-1}=(ks-1)^{-[r/k]}$, and if $k>r/2$ then $[r/k]=1$, we can
unify the above expression of $B(s)$ as
$$
B(s)
=\frac{1}{r}(-1)^{k-1}D_{r-k}(k)(ks-1)^{-[r/k]},
$$
where 
\begin{align}
D_{r-k}(k)=\left\{
\begin{array}{lll}
C_{r-k}(k) & {\rm if} & k\leq r/2,\\
\zeta_{r-k}(1/k) & {\rm if} &  k>r/2.
\end{array}
\right.
\end{align}
Since $[(r-j)/k]\leq [r/k]$ for any $j\geq 1$, the main contribution on the right-hand
side of \eqref{zeta_asymp} would be coming from $B(s)$, and the part of $A(s)$
consisting of only $j$ satisfying $[(r-j)/k]=[r/k]$.    That is, it would be that
\begin{align}\label{zeta_sim}
\zeta_r(s)\sim C_r(k)(ks-1)^{-[r/k]},
\end{align}
where
\begin{align}\label{C_r_k_def}
C_r(k)= \frac{1}{r}\left\{{\sum_j}^* (-1)^{j-1}\zeta\left(\frac{j}{k}\right)C_{r-j}(k)
+(-1)^{k-1}D_{r-k}(k)\right\},
\end{align}
and here, the symbol $\sum^*$ stands for the summation on $j$ satisfying
$j\leq r-k$, $j\neq k$ and $[(r-j)/k]=[r/k]$. 

To establish \eqref{zeta_sim} rigorously, it is necessary to check that
$C_r(k)\neq 0$.    This can be seen by observing the signature.    
By induction assumption, $C_{r-j}(k)$ is non-zero, and the signature of $C_{r-j}(k)$ coincides with the 
signature of $(-1)^{r-j+[(r-j)/k]}$.
Therefore the signature of $(-1)^{k-1}C_{r-k}(k)$ is 
$$(-1)^{k-1+(r-k+[(r-k)/k])}=(-1)^{r+[r/k]},$$
while the signature of $(-1)^{j-1}\zeta(j/k)C_{r-j}(k)$ is
$$
(-1)^{(j-1)+1+(r-j+[(r-j)/k]) }=(-1)^{r+[r/k]}
$$
for any $j$ which appears in the sum $\sum^*$
(because of the condition on $\sum^*$, and the fact $\zeta(j/k)<0$).
Moreover, if $k>r/2$, then $1/k<1/(r-k)$ so, by Theorem \ref{Th<1/r}, the signature 
of $(-1)^{k-1}\zeta_{r-k}(1/k)$ is $(-1)^{(k-1)+(r-k)}=(-1)^{r-1}$, and this is 
further equal to 
$(-1)^{r+[r/k]}$ because now $[r/k]=1$.   
It follows that the signatures of all terms on the right-hand side of 
\eqref{C_r_k_def} are the same.    Therefore obviously $C_r(k)\neq 0$, and its signature
coincides with $(-1)^{r+[r/k]}$.    The theorem is now proved.
\end{proof}

Moreover, we can determine the explicit values of the constants $C_r(k)$.

\begin{theorem}\label{Th-coeff}
We have
\begin{align}\label{Th_formula2}
C_r(r)=(-1)^{r-1}\frac{1}{r} \qquad (r\geq 1),
\end{align}
\begin{align}\label{Th_formula3}
C_r(k)=\frac{(-1)^{k-1}}{k}\zeta_{r-k}\left(\frac{1}{k}\right) \qquad 
(r/2< k\leq r-1),
\end{align}
\begin{align}\label{Th_formula4}
C_r(1)=\frac{1}{r!} \qquad (r\geq 1),
\end{align}
and, for any $k\geq 2$, 
\begin{align}\label{Th_formula5}
&C_{kr}(k)=\frac{(-1)^{(k-1)r}}{k^r\cdot r!} \qquad (r\geq 1),\\
&C_{kr+\ell}(k)=\frac{(-1)^{(k-1)r}}{k^r\cdot r!}\zeta_{\ell}\left(\frac{1}{k}\right)
\qquad (r\geq 1, 1\leq \ell\leq k-1). \label{Th_formula6}
\end{align}
\end{theorem}

\begin{remark}
For any fixed $k\geq 2$, formulas \eqref{Th_formula5} and \eqref{Th_formula6} show
that $C_r(k)$, as a function in $r$, has a kind of ``periodicity'' mod $k$.
\end{remark}

\begin{remark}
Formulas \eqref{Th_formula4}, \eqref{Th_formula5} and \eqref{Th_formula6} exhausts
all the cases of $C_r(k)$, $k,r\in\mathbb{N}$, $1\leq k\leq r$.     However we still
prefer to include \eqref{Th_formula2} and \eqref{Th_formula3} in the statement of
Theorem \ref{Th-coeff}, because they themselves are elegant formulas, and they are
necessary in the proof of \eqref{Th_formula5} and \eqref{Th_formula6}.
\end{remark}

\begin{proof}[Proof of Theorem \ref{Th-coeff}]
The first formula \eqref{Th_formula2} was already shown in the proof of Theorem
\ref{Th-asymp}. To prove the remaining formulas, we use \eqref{C_r_k_def}.

Let $r/2< k\leq r-1$.   Then $D_{r-k}(k)=\zeta_{r-k}(1/k)$, hence 
\eqref{C_r_k_def} is
\begin{align}\label{C_r_k_def2}
C_r(k)
=\frac{1}{r}\left\{{\sum_j}^* (-1)^{j-1}\zeta\left(\frac{j}{k}\right)C_{r-j}(k)
+(-1)^{k-1}\zeta_{r-k}\left(\frac{1}{k}\right)\right\},
\end{align}
where the summation $\sum^*$ runs over $1\leq j\leq r-k$.
First consider the case $k=r-1$.   Then
\begin{align}
C_r(r-1)=\frac{1}{r}\left\{\zeta\left(\frac{1}{r-1}\right)C_{r-1}(r-1)
+(-1)^{r-2}\zeta\left(\frac{1}{r-1}\right)\right\}.
\end{align}
Since $C_{r-1}(r-1)=(-1)^{r-2} (r-1)^{-1}$ by \eqref{Th_formula2}, the above is equal to
$$
\frac{(-1)^{r-2}}{r}\cdot \left(\frac{1}{r-1}+1\right)\zeta\left(\frac{1}{r-1}\right)
=\frac{(-1)^{r-2}}{r-1}\zeta\left(\frac{1}{r-1}\right),
$$
which is \eqref{Th_formula3} for $k=r-1$.

Now we show \eqref{Th_formula3} by induction on $r-k$. 
We may apply induction assumption to $C_{r-j}(k)$ on the right-hand side of
\eqref{C_r_k_def2} for $1\leq j\leq r-k-1$, because $(r-j)/2< k\leq (r-j)-1$ for
those $j$ and $(r-j)-k < r-k$.
Hence
we apply \eqref{Th_formula3} (for $1\leq j\leq r-k-1$) and \eqref{Th_formula2}
(for $j=r-k$) to the right-hand side of \eqref{C_r_k_def2} to get
\begin{align}
C_r(k)&=\frac{1}{r}\left\{\sum_{j=1}^{r-k-1}(-1)^{j-1}\zeta\left(\frac{j}{k}\right)
\frac{(-1)^{k+1}}{k}\zeta_{r-k-j}\left(\frac{1}{k}\right)
+(-1)^{r-k-1}\zeta\left(\frac{r-k}{k}\right)(-1)^{k-1}\frac{1}{k}\right\}\\
&\;+\frac{(-1)^{k-1}}{r}\zeta_{r-k}\left(\frac{1}{k}\right)\notag\\
&=\frac{(-1)^{k+1}}{rk}\sum_{j=1}^{r-k}(-1)^{j-1}
\zeta_{r-k-j}\left(\frac{1}{k}\right)\zeta\left(\frac{j}{k}\right)
+\frac{(-1)^{k-1}}{r}\zeta_{r-k}\left(\frac{1}{k}\right).\notag
\end{align}
Since the sum on the right-hand side is equal to $(r-k) \zeta_{r-k}(1/k)$ by \eqref{zeta_id}, 
we obtain
$$
C_r(k)=(-1)^{k+1}\left(\frac{r-k}{rk}+\frac{1}{r}\right)
\zeta_{r-k}\left(\frac{1}{k}\right)
=\frac{(-1)^{k+1}}{k}\zeta_{r-k}\left(\frac{1}{k}\right),
$$
which is \eqref{Th_formula3}.

Next consider $C_r(1)$.    The case $r=1$ is included in \eqref{Th_formula2}.
Assume $r\geq 2$.    Then \eqref{C_r_k_def} implies $C_r(1)=r^{-1}C_{r-1}(1)$, from
which \eqref{Th_formula4} immediately follows.

Lastly we prove \eqref{Th_formula5} and \eqref{Th_formula6}.
First we notice that the case $r=1$ of \eqref{Th_formula5} and \eqref{Th_formula6}
is included in \eqref{Th_formula2} and \eqref{Th_formula3}, respectively.
Therefore now assume $r\geq 2$, and prove the formulas by induction on $r$.

From \eqref{C_r_k_def} we have
\begin{align}\label{C_r_k_def3}
C_{kr+\ell}(k)= \frac{1}{kr+\ell}\left\{{\sum_j}^* (-1)^{j-1}\zeta\left(\frac{j}{k}\right)
C_{kr+\ell-j}(k)
+(-1)^{k-1}C_{kr+\ell-k}(k)\right\},
\end{align}
because $k\leq (kr+\ell)/2$.    The summation $\sum^*$ runs over all $1\leq j\leq \ell$
(because $[(kr+\ell-j)/k]=[(kr+\ell)/k]=r$ should be satisfied).
In particular, when $\ell=0$ this sum is empty, hence
\begin{align}
C_{kr}(k)=\frac{(-1)^{k-1}}{kr}C_{k(r-1)}(k).
\end{align}
Therefore recursively we obtain
$$
C_{kr}(k)=\frac{(-1)^{k-1}}{kr}\cdot \frac{(-1)^{k-1}}{k(r-1)}\cdots
\frac{(-1)^{k-1}}{2k}C_k(k)
=\frac{(-1)^{(k-1)(r-1)}}{k^{r-1}r!}\cdot \frac{(-1)^{k-1}}{k}
=\frac{(-1)^{(k-1)r}}{k^r\cdot r!},
$$
which is \eqref{Th_formula5}.

Assume $\ell\geq 1$.    Here we also adopt the induction on $\ell$.   Since 
$\ell-j<\ell$, we use the induction (on $\ell$) assumption to $C_{kr+\ell-j}(k)$ and
the induction (on $r$) assumption to $C_{kr+\ell-k}(k)$ on the right-hand side of
\eqref{C_r_k_def3}.    Then
\begin{align}
C_{kr+\ell}(k)=\frac{1}{kr+\ell}\left\{\sum_{j=1}^{\ell}(-1)^{j-1}\zeta\left(\frac{j}{k}\right)
\frac{(-1)^{(k-1)r}}{k^r\cdot r!}\zeta_{\ell-j}\left(\frac{1}{k}\right)
+(-1)^{k-1}
\frac{(-1)^{(k-1)(r-1)}}{k^{r-1}(r-1)!}\zeta_{\ell}\left(\frac{1}{k}\right)\right\}.
\end{align}
The sum on the right-hand side is
$$
\frac{(-1)^{(k-1)r}}{k^r\cdot r!}\cdot \ell\zeta_{\ell}\left(\frac{1}{k}\right)
$$
by \eqref{zeta_id}, hence
\begin{align}
C_{kr+\ell}(k)=\frac{1}{kr+\ell}\cdot \frac{(-1)^{(k-1)r}}{k^{r-1}(r-1)!}
\left(\frac{\ell}{kr}+1\right)\zeta_{\ell}\left(\frac{1}{k}\right)
=\frac{(-1)^{(k-1)r}}{k^r\cdot r!}\zeta_{\ell}\left(\frac{1}{k}\right),
\end{align}
which is \eqref{Th_formula6}.
\end{proof}

Now we consider the distribution of real zeros of $\zeta_r(s)$.   By Theorem
\ref{Th<1/r} we know that $\zeta_r(s)\neq 0$ for $0\leq s<1/r$.    However,
the graphs for $\zeta_r(s)$ for $2\leq r\leq 6$ show the existence of zeros
in other intervals.    Denote by $I_r(k)$ the number of IAZs of $\zeta_r(s)$
in the interval $(1/k,1/(k-1))$.    From the graphs we may observe that all zeros
in the graphs seem simple, and
\begin{align*}
& I_2(2)=1,\\
& I_3(3)=1,I_3(2)=1,\\
& I_4(4)=1, I_4(3)=1, I_4(2)=2,\\
& I_5(5)=1, I_5(4)=1, I_5(3)=1, I_5(2)=2,\\
& I_6(6)=1, I_6(5)=1, I_6(4)=1, I_6(3)=2, I_6(2)=3.
\end{align*}
In the next subsection we will present the graphs of $\zeta_r(s)$, $7\leq r\leq 10$.
From those graphs we further observe:
\begin{align*}
& I_7(7)=1, I_7(6)=1, I_7(5)=1, I_7(4)=1, I_7(3)=2, I_7(2)=3,\\
& I_8(8)=1, I_8(7)=1, I_8(6)=1, I_8(5)=1, I_8(4)=2, I_8(3)=2, I_8(2)=4,\\
& I_9(9)=1, I_9(8)=1, I_9(7)=1, I_9(6)=1, I_9(5)=1, I_9(4)=2, I_9(3)=3, I_9(2)=4,\\
& I_{10}(10)=1, I_{10}(9)=1, I_{10}(8)=1, I_{10}(7)=1, I_{10}(6)=1, I_{10}(5)=2, 
I_{10}(4)=2, I_{10}(3)=3, I_{10}(2)=5.
\end{align*}

 Based on these data, here we propose the following conjecture.
 
\begin{conjecture}\label{conj-IAZ}
For any $r\geq 2$, all IAZs of $\zeta_r(s)$ are simple, and $I_r(k)=[r/k]$
$(2\leq k\leq r)$.
\end{conjecture}
 
The above data give a strong evidence for this conjecture, but so far we have
not found any rigorous proof of the conjecture. 
 
Let $IAZ(r)$ be the total number of IAZs, that is, the number of all real zeros
of $\zeta_r(s)$ in the interval $(0,1)$.

\begin{theorem}
If Conjecture \ref{conj-IAZ} is true, then we have
\begin{align}
IAZ(r)=\sum_{k=2}^r \left[\frac{r}{k}\right]=r\log r-2(1-\gamma)r+O(r^{1/2}),
\end{align}
where $\gamma=0.577215\ldots$ is Euler's constant.
\end{theorem}

\begin{proof}
The first equality is a direct consequence of the conjecture.    The second
equality follows from
\begin{align}\label{F-d}
\sum_{k=2}^r \left[\frac{r}{k}\right]=
\sum_{k=1}^r \left[\frac{r}{k}\right]-r
=\sum_{k=1}^r \sum_{\substack{\ell\leq r \\ \ell\equiv 0 ({\rm mod}\; k)}}1-r
=\sum_{\ell\leq r}d(\ell)-r,
\end{align}
where $d(\ell)$ denotes the number of positive divisors of $\ell$, and the known result
$$
\sum_{\ell=1}^r d(\ell)=r\log r+(2\gamma-1)r+O(r^{1/2})
$$
(\cite[Theorem 3.3]{Apos76}).
\end{proof}

\begin{remark}
The above error term $O(r^{1/2})$ is not best-possible.    It is conjectured that
the estimate $O(r^{1/4+\varepsilon})$ would hold for any $\varepsilon>0$, 
and the best known result is $O(r^{131/416+\varepsilon})$ due to
Huxley \cite{Hux03}.
\end{remark}

\begin{remark}
Write $F(r)=\sum_{k=2}^r [r/k]$, and let
$$
\Delta_{IAZ}(r)=IAZ(r)-IAZ(r-1), \qquad \Delta_F(r)=F(r)-F(r-1).
$$
If Conjecture \ref{conj-IAZ} is true, then $\Delta_{IAZ}(r)=\Delta_F(r)$, so it is
interesting to observe the behavior of $\Delta_F(r)$.
From \eqref{F-d} we see that
$\Delta_F(r)=d(r)-1$.    In particular, 
$\Delta_F(p)=1$ for any prime number $p$.
Further, $\Delta_F(r)$ is dominantly odd; it is even if and only if $r$ is a
square (because if the decomposition of $r$ into prime factors is
$r=p_1^{a_1}p_2^{a_2}\cdots p_k^{a_k}$, then $d(r)=(a_1+1)(a_2+1)\cdots (a_r+1)$).
\end{remark}

\section{Further examples}

Here we present the graphs of $\zeta_r(s)$, $7\leq r\leq 10$, $s\in [0,1]$.
We find that the behavior of those graphs agrees with our Conjecture \ref{conj-IAZ}.
 
 The seven-fold zeta-function $\zeta_7(s)$ has seven asymptotes: $\Re s=1, 1/2, 1/3, 1/4, 1/5, 1/6$ and $1/7$, and has nine IAZs for \( s\in [0,1]\) (Fig \ref{Fig9}):

\begin{itemize}[noitemsep]
    \item One IAZ $\in (1/7,1/6)$: \(\zeta_7(0,152170...)\approx 0.\)
    
    \item One IAZ $\in (1/6,1/5)$: \(\zeta_7(0,178811...)\approx 0. \)
    \item One IAZs $\in (1/5,1/4)$: \(\zeta_7(0,217987...) \approx0 \).
    \item One IAZs $\in (1/4,1/3)$: \(\zeta_7(0,298653...) \approx0 \).
    \item Two IAZs $\in (1/3,1/2)$: \(\zeta_7(0,365596..)\approx 0 \) and \(\zeta_7(0,442820...)\approx 0 \).
    
    \item Three IAZs $\in (1/2,1)$: \(\zeta_7(0,605776...)\approx 0 \), \(\zeta_7(0,736271...)\approx 0 \) and
\(\zeta_7(0,868958...)\approx 0.\) 
\end{itemize} 

Also \(\zeta_7(s)\) has one minimum and two maximums for \( s\in [0,1]\):

\begin{itemize}[noitemsep]
    \item One maximum between vertical asymptotes $1/3$ and $1/2$: \(\zeta_7(0,412055...)\approx  +4,875899...\)
    \item One minimum \(\zeta_7(0,663498...)\approx  -1,927124...\) and one maximum \(\zeta_7(0,847083...)\approx  +21,72816...\)
between vertical asymptotes $1/2$ and 1.
\end{itemize}

\begin{figure}[h]
\centering
\includegraphics{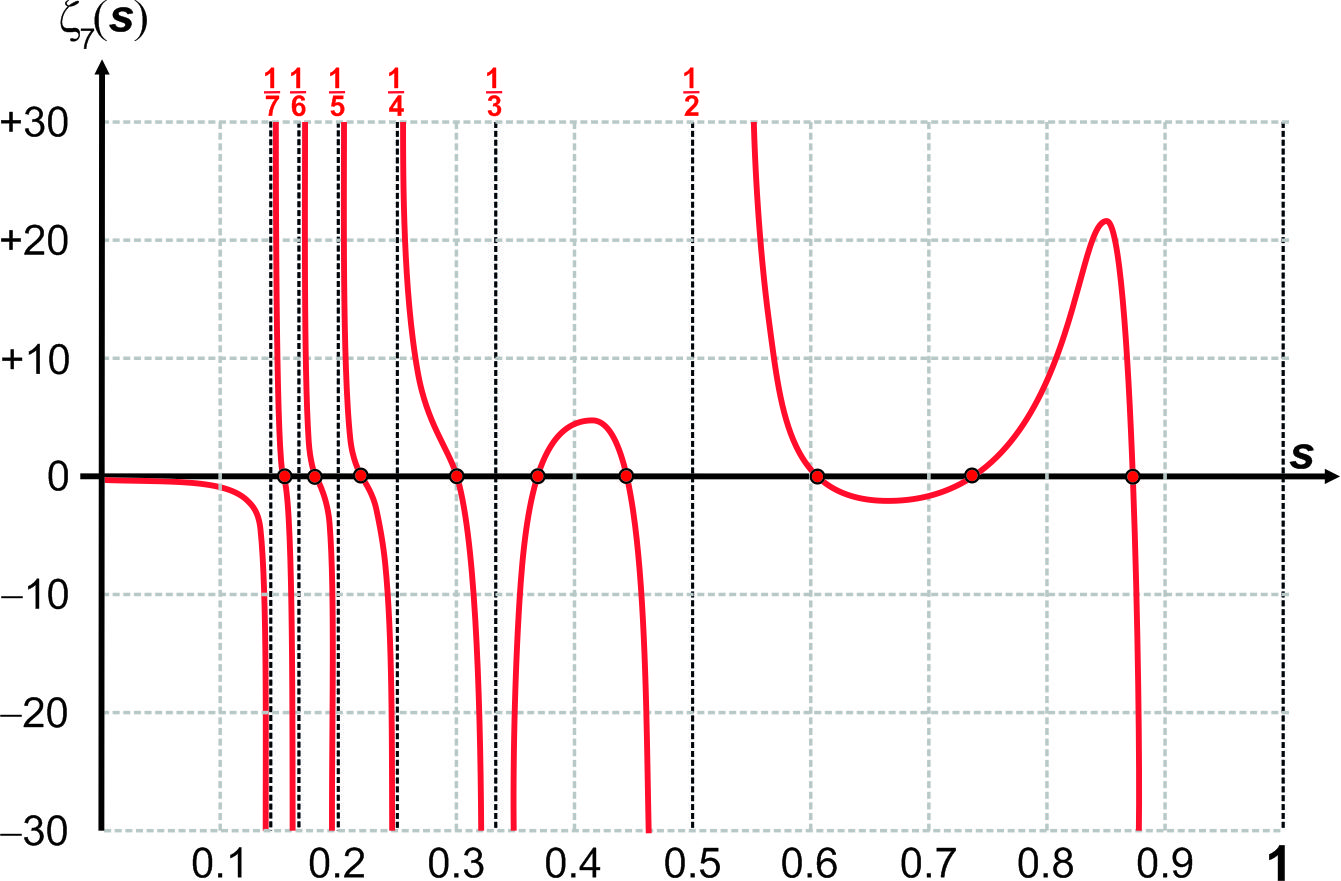}
\caption{The seven-fold zeta-function for $s\in [0,1]$}
\label{Fig9}
\end{figure}

The eight-fold zeta-function $\zeta_8(s)$ has eight asymptotes: $\Re s=1, 1/2, 1/3, 1/4, 1/5, 1/6, 1/7$ and $1/8$, and has twelve IAZs in \( s\in [0,1]\) (Fig \ref{Fig10}):

\begin{itemize}[noitemsep]
    \item One IAZ $\in (1/8,1/7)$: \(\zeta_8(0,132134...)\approx 0\)
    
    \item One IAZ $\in (1/7,1/6)$: \(\zeta_8(0,151738...)\approx 0.\)
    \item One IAZs $\in (1/6,1/5)$: \(\zeta_8(0,178822...) \approx0 \).
    \item One IAZs $\in (1/5,1/4)$: \(\zeta_8(0,218978...) \approx0 \).
    \item Two IAZs $\in (1/4,1/3)$: \(\zeta_8(0,266060...)\approx 0 \) and \(\zeta_8(0,295787..)\approx 0 \).
    \item Two IAZs $\in (1/3,1/2)$: \(\zeta_8(0,392752...)\approx 0 \) and \(\zeta_8(0,437321...)\approx 0 \).
    
    \item Four IAZs $\in (1/2,1)$: \(\zeta_8(0,538144...)\approx 0 \), \(\zeta_8(0,650658...)\approx 0 \) ,
\(\zeta_8(0,766794...)\approx 0 \) and \(\zeta_8(0,883665...)\approx 0 \).
\end{itemize} 

Also \(\zeta_8(s)\) has four minimums and one maximum for  \( s\in [0,1]\) :

\begin{itemize}[noitemsep]
    \item One minimun between vertical asymptotes $1/4$ and $1/3$: \(\zeta_8(0,277976...)\approx -2,253261...,\)
    \item One minimum between  vertical asymptotes $1/3$ and $1/2$: \(\zeta_8(0,420354...)\approx  -2,635752...,\) 
    
    \item Two minimus and one maximum between vertical asymptotes $1/2$ and 1: 
the first minimum at \( \zeta_8(0,551113...)\approx -12,188697...\), the maximum at \(\zeta_8(0,719417...)\approx  +1,334459...\), and the second minimum at \(\zeta_8(0,867892...)\approx  -45,821285....\)
\end{itemize}

\begin{figure}[h]
\centering
\includegraphics{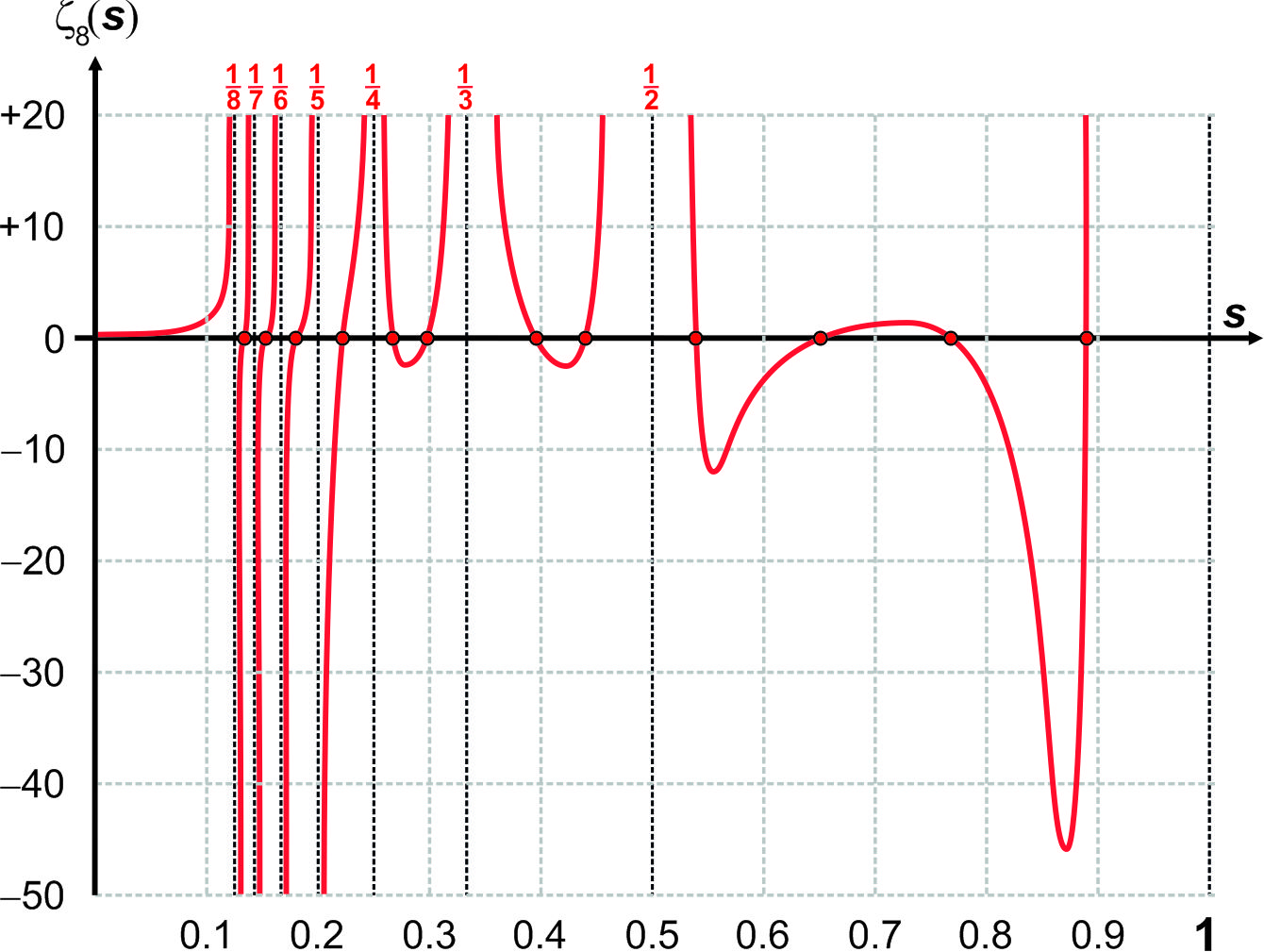}
\caption{The eight-fold zeta-function for $s\in [0,1]$}
\label{Fig10}
\end{figure}

Figure \ref{Fig11} shows the behavior of the nine-fold zeta-function, which first produces three
IAZs on the inter asymptotic interval \( s\in [1/3, 1/2]\).


\begin{figure}[H]
\centering
\includegraphics{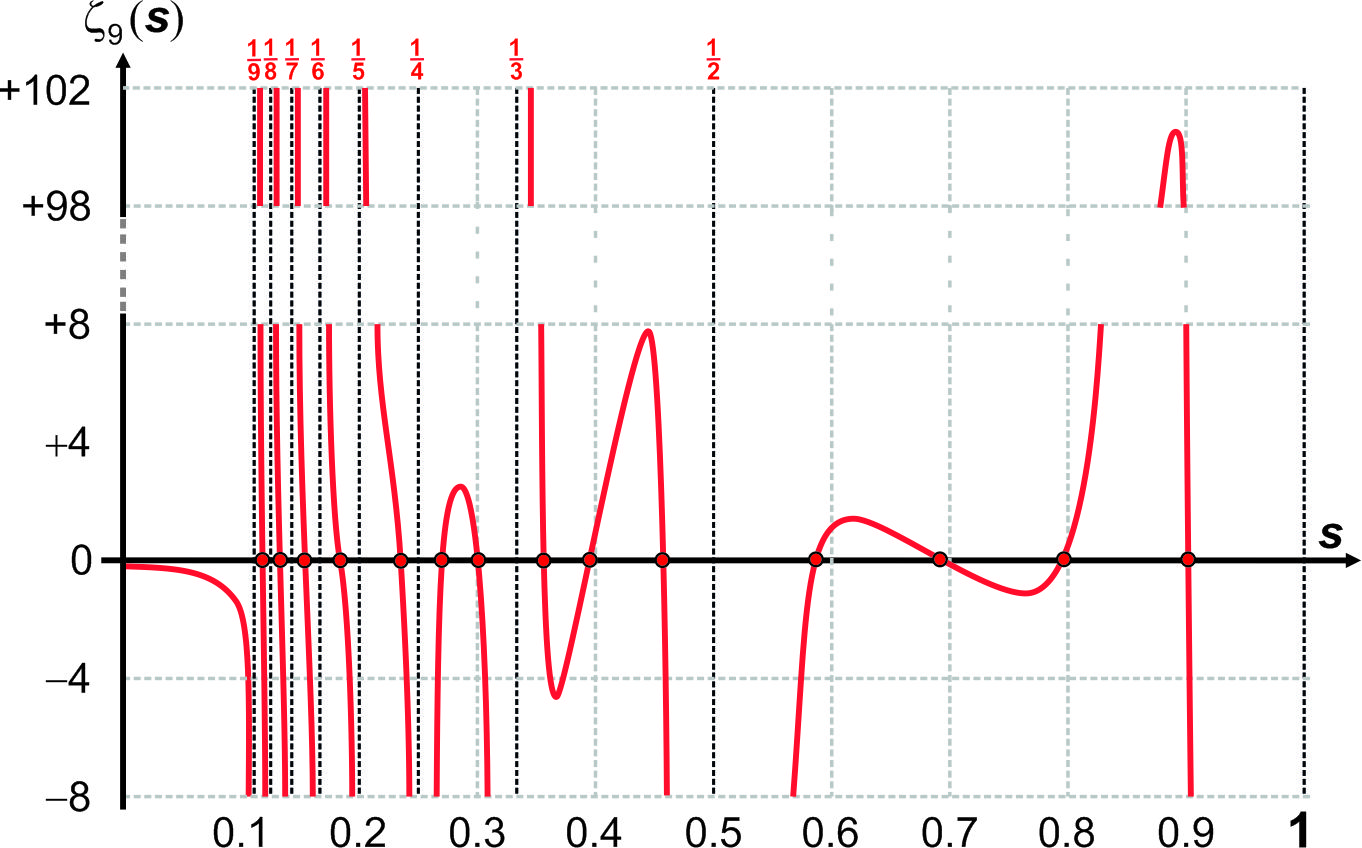}
\caption{The nine-fold zeta-function for $s\in [0,1]$}
\label{Fig11}
\end{figure}

Figure \ref{Fig12} shows the behavior of the ten-fold zeta-function, which first 
produces five IAZs at the inter-asymptotic interval \( s\in [1/2,1]\)  and
also, two IAZs at the inter asymptotic interval \( s\in [1/5,1/4]\).

%

\begin{figure}[H]
\centering
\includegraphics{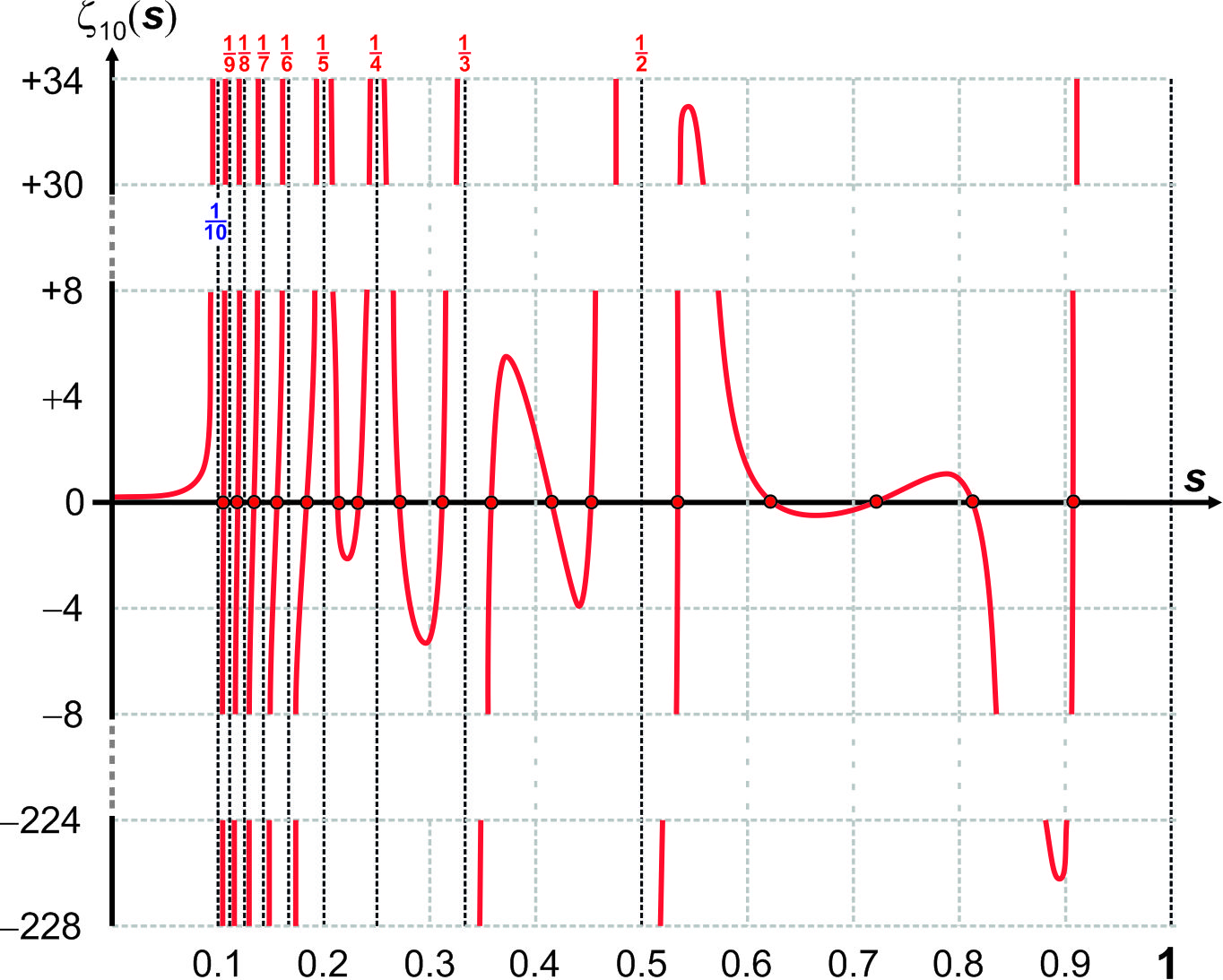}
\caption{The ten-fold zeta-function for $s\in [0,1]$}
\label{Fig12}
\end{figure}

\

\end{document}